\documentclass[11pt,reqno]{amsart}
\usepackage[leqno]{amsmath}
\usepackage{amscd,amssymb,palatino}
\usepackage[english]{babel}
\usepackage{caption}
\usepackage{etex}
\usepackage[leqno]{amsmath}
\usepackage{amssymb}
\usepackage{color}
\usepackage{shadow}
\usepackage{epsfig}
\usepackage{epic}
\usepackage[dvipsnames]{xcolor}
\usepackage{graphics}
\usepackage{graphicx}
\usepackage{psfrag}
\usepackage{url}
\usepackage{hyperref}
\usepackage{comment}

\usepackage[capitalize]{cleveref} 
\usepackage{tensor}
\usepackage{enumitem}

\usepackage{amsthm, amscd, amsfonts, amssymb} 
\usepackage{mathtools} 
\usepackage{color}

\theoremstyle{plain}
\newtheorem{theoremIntro}{Theorem}

\newtheorem{lemmaIntro}[theoremIntro]{Lemma}

\newtheorem{proposition}{Proposition}[section]
\newtheorem{theorem}[proposition]{Theorem}

\newtheorem{lemma}[proposition]{Lemma}

\theoremstyle{definition}
\newtheorem{definition}[proposition]{Definition}

\theoremstyle{remark}
\newtheorem{remark}[proposition]{Remark}

\newcommand\restr[2]{{
  \left.\kern-\nulldelimiterspace
  #1
  \vphantom{\big|}
  \right|_{#2}
  }}

\DeclarePairedDelimiterX{\innerp}[2]{\langle}{\rangle}{#1,#2}
\newcommand{\dif}{\mathop{}\!d}

\DeclareMathOperator{\hodgestar}{\star}
\DeclareMathOperator{\identity}{Id}

\newcommand{\emptyslot}{-}

\newcommand{\R}{\mathbb{R}}

\usepackage{setspace}
\usepackage[style=numeric, url=false]{biblatex}
\usepackage{graphicx}

\usepackage{tikz-cd}
\usepackage{circuitikz}
\usetikzlibrary{shapes.geometric, arrows}

\tikzstyle{mytheorembox} = [draw=vdgreen, fill=blue!20, very thick, rectangle, rounded corners, inner sep=10pt, inner ysep=15pt]
\tikzstyle{mytheoremfancytitle} =[fill=vdgreen, text=white]

\definecolor{vdblue}{rgb}{0,0,.3}
\definecolor{dblue}{rgb}{0,0,.7}
\definecolor{lblue}{rgb}{.3,.3,1}
\definecolor{vlblue}{rgb}{.7,.7,1}
\definecolor{vvlblue}{rgb}{.9,.9,1}

\definecolor{vdred}{rgb}{.3,0,0}
\definecolor{dred}{rgb}{.7,0,0}
\definecolor{lred}{rgb}{1,.3,.3}
\definecolor{vlred}{rgb}{1,.7,.7}

\definecolor{vdgreen}{rgb}{0,.2,0}
\definecolor{dgreen}{rgb}{0,.4,0}
\definecolor{lgreen}{rgb}{.3,1,.3}
\definecolor{vlgreen}{rgb}{.7,1,.7}

\definecolor{lyellow}{rgb}{1,1,.3}

\definecolor{gray1}{rgb}{0.22,0.22,0.22}
\definecolor{gray2}{rgb}{0.28,0.28,0.28}
\definecolor{gray3}{rgb}{0.36,0.36,0.36}
\definecolor{gray4}{rgb}{0.44,0.44,0.44}
\definecolor{gray5}{rgb}{0.52,0.52,0.52}
\definecolor{gray6}{rgb}{0.6,0.6,0.6}
\definecolor{gray7}{rgb}{0.68,0.68,0.68}
\definecolor{gray8}{rgb}{0.76,0.76,0.76}

\definecolor{color1}{rgb}{1,0,0}
\definecolor{color2}{rgb}{0.98,0,0.816}
\definecolor{color3}{rgb}{0.717,0,1}
\definecolor{color4}{rgb}{0,0,1}
\definecolor{color5}{rgb}{0,1,1}
\definecolor{color6}{rgb}{0,1,0}
\definecolor{color8}{rgb}{1,1,0}
\definecolor{color7}{rgb}{1,0.651,0}

\begin{document}
\title{Turing complete Navier-Stokes steady states via\\ cosymplectic geometry}
\author{Søren Dyhr}
\address{
	Søren Dyhr, Laboratory of Geometry and Dynamical Systems, Universitat Po\-li\-tèc\-ni\-ca de Ca\-ta\-lunya, 08028 Barcelona, and
	Centre de Recerca Matemàtica, Edifici C, Campus Bellaterra, 08193 Bellaterra, Barcelona.
}
\email{soren.dyhr@upc.edu}

\author{Ángel González-Prieto}
\address{ \'Angel Gonz\'alez-Prieto, Facultad de CC.\ Matem\'aticas, Universidad Complutense de Madrid, 28040 Madrid, Spain, and Instituto de Ciencias Matem\'aticas,
28049 Madrid, Spain.}
\email{angelgonzalezprieto@ucm.es}
\author{Eva Miranda}\address{Eva Miranda,
Laboratory of Geometry and Dynamical Systems, Department of Mathematics, EPSEB, Universitat Polit\`{e}cnica de Catalunya-IMTech, Barcelona, and Centre de Recerca Matem\`{a}tica, Edifici C, Campus de Bellaterra, 08193 Bellaterra, Barcelona.
 }
 \email{eva.miranda@upc.edu}
 \author{Daniel Peralta-Salas} \address{Daniel Peralta-Salas, Instituto de Ciencias Matem\'aticas, Consejo Superior de Investigaciones Cient\'ificas,
28049 Madrid, Spain.}
\email{dperalta@icmat.es}

\thanks{S.\ Dyhr received the support of a fellowship from the ”la Caixa” Foundation (ID 100010434) with fellowship code LCF/BQ/DI23/11990085. Á.\ González-Prieto is supported by the Spanish State Research Agency project PID2021-124440NB-I00. E.\ Miranda is supported by the Catalan Institution for Research and Advanced Studies via an ICREA Academia Prize 2021. D.\ Peralta-Salas is supported by the grants RED2022-134301-T and PID2022-136795NB-I00 funded by MCIN/AEI/10.13039/501100011033.  S.\ Dyhr and E.\ Miranda are supported by the Spanish State Research Agency project PID2023-146936NB-I00 funded by MICIU/AEI/10.13039/501100011033. \'A.\ Gonz\'alez-Prieto and D.\ Peralta-Salas are supported by the project CEX2023-001347-S. 
The last three authors are supported by the Bilateral AEI-DFG project: Celestial Mechanics, Hydrodynamics, and Turing Machines with reference codes PCI2024-155042-2 and PCI2024-155062-2. All authors are partially supported by the project “Computational, dynamical and geometrical complexity in fluid dynamics” (COMPLEXFLUIDS), Ayudas Fundación BBVA a Proyectos de Investigación Científica 2021.
This work is supported by the Spanish State Research Agency, through the Severo Ochoa and María de Maeztu Program for Centers and Units of Excellence in R\&D (CEX2020-001084-M). We thank CERCA Programme/Generalitat de Catalunya for institutional support.}

\begin{abstract} 
In this article, we construct stationary solutions to the Navier-Stokes equations on certain Riemannian $3$-manifolds that exhibit Turing completeness, in the sense that they are capable of performing universal computation. This universality arises on manifolds admitting nonvanishing harmonic 1-forms, thus showing that computational universality is not obstructed by viscosity, provided the underlying geometry satisfies a mild cohomological condition. The proof makes use of a correspondence between nonvanishing harmonic $1$-forms and cosymplectic geometry, which extends the classical correspondence between Beltrami fields and Reeb flows on contact manifolds.
\end{abstract}

\maketitle

\section*{Significance Statement}

The Navier-Stokes equations govern the behaviour of viscous fluids and are at the heart of classical physics, engineering, and applied mathematics. Despite their ubiquity, fundamental questions about their solutions remain unresolved. This work establishes, for the first time, the existence of stationary fluid flows governed by the Navier-Stokes equations that can simulate any computation performable by a universal Turing machine. These Turing complete solutions are constructed on geometrically natural 3-manifolds using tools from cosymplectic geometry and the theory of harmonic forms. The result reveals that undecidability—a hallmark of computational complexity—can arise even in the steady state of viscous fluid motion, reshaping our understanding of what fluids can (and cannot) compute.

\section{Introduction}

In 2016, Terence Tao \cite{taoblowup} (see also \cite{taosingularities}) proposed an original and provocative approach to address the celebrated blow-up problem for the Navier-Stokes equations: instead of directly proving regularity or blow-up, one might attempt to encode universal computation into the Navier-Stokes dynamics. If successful, this would imply that the long-term behaviour of solutions is as unpredictable as the halting problem, a hallmark of undecidability and Turing completeness. Such an approach reframes the regularity question within the broader context of computational complexity and logic, potentially explaining why the problem has resisted classical  techniques.

Motivated by this perspective, recent research has shown the existence of Turing complete solutions to the Euler equations on some Riemannian manifolds, both for the stationary and the time-dependent equations \cite{CMPP2}, \cite{CMP3}, \cite{mirror}, \cite{fluidcomputer} and \cite{timedependent}. However, the case of the Navier-Stokes equations has remained open until now. A particularly successful approach for the construction of Turing complete Euler flows is the correspondence, unveiled by Sullivan and developed by Etnyre and Ghrist \cite{etnyreContactTopologyHydrodynamics2000}, between stationary solutions to the Euler equations and Reeb vector fields of contact structures on three-manifolds. This relationship offers an interesting geometric view through which ideal fluids can be interpreted. Unfortunately, the aforementioned correspondence does not allow one to construct solutions to the Navier-Stokes equations.

In this article, we aim to construct Turing complete solutions to the Navier-Stokes equations on some Riemannian $3$-manifolds, making use of some geometric ideas that extend the Etnyre-Ghrist correspondence. Specifically, using the machinery of cosymplectic geometry, we construct harmonic vector fields that simulate any Turing machine. These harmonic fields are stationary solutions to the Navier-Stokes equations, providing the Turing completeness of these flows. Notice that, as a consequence, the resulting flow has undecidable long-term behaviour.

To state our main theorem in a precise way, we next introduce some notation. Let \( (M, g) \) be a smooth compact Riemannian manifold without boundary. This structure defines the so-called Hodge Laplacian $\Delta = dd^* + d^*d \colon \Omega^k(M) \to \Omega^k(M)$ on $k$-differential forms, where $d^* \colon \Omega^k(M) \to \Omega^{k-1}(M)$ is the adjoint operator of the usual exterior differential $d \colon \Omega^k(M) \to \Omega^{k+1}(M)$. Furthermore, this Laplacian operator can be extended to vector fields $X \in \mathfrak{X}(M)$ by setting $\Delta X := (\Delta X^\flat)^\sharp$, where $X^\flat = g(X, \emptyslot)$ is the dual $1$-form with inverse $\sharp$.


The Navier-Stokes equations on $(M,g)$ describe the motion of an incompressible viscous fluid.
Given a viscosity coefficient $\nu > 0$, the system seeks (possibly non-autonomous) vector fields $X$ and pressure functions $p$ satisfying the equations:
\begin{equation}\label{eq:NS}
\begin{cases}
\displaystyle\frac{\partial X}{\partial t} + \nabla_X X - \nu \Delta X = -\nabla p\,, \\[5pt]
\operatorname{div} X = 0\,,
\end{cases}
\end{equation}
where all differential operators are defined with respect to the Riemannian metric $g$.
The first equation reflects Newton’s second law, incorporating advection, viscous dissipation, and pressure forces. The second one enforces the incompressibility condition.

In this direction, a smooth Riemannian manifold \( (M, g) \) will be said to be \emph{Hodge-admissible} if it admits a \emph{nowhere vanishing harmonic vector field}, that is, a smooth vector field \( X \in \mathfrak{X}(M) \) such that its dual $1$-form \( \alpha := X^\flat \) satisfies
\[
d\alpha = 0 \quad \text{and} \quad d^* \alpha = 0,
\]
and \( \alpha \) is nowhere zero on \( M \). Equivalently, \( \alpha \) lies in the kernel of the Hodge Laplacian \( \Delta = dd^* + d^*d \) and is nowhere vanishing. It is worth noticing that this condition is stronger than simply requiring the existence of (nontrivial) harmonic 1-forms: while the latter is a topological condition governed by the first cohomology group \( H^1(M) \), the Hodge-admissibility condition imposes geometric and analytic constraints on the structure of the manifold and its vector fields. It is well known that harmonic vector fields are stationary solutions to the Navier-Stokes equations, independently of the value of the viscosity $\nu\geq 0$. 

Now we are ready to state our main result. A straightforward but remarkable consequence of this theorem is the existence of viscous fluid flows exhibiting undecidable trajectories.

\begin{theoremIntro}\label{T:main} Given any Hodge-admissible Riemannian $3$-manifold $(M,g)$, there exists a (generally not small) deformation $g_\epsilon$ of the metric $g$ and a solution of the stationary Navier-Stokes equations associated to $g_\epsilon$ that is Turing complete, for any value of the viscosity $\nu\geq 0$.
\end{theoremIntro}

The proof of Theorem~\ref{T:main} exploits a connection between harmonic forms and weak cosymplectic structures.
We recall that a \emph{weak cosymplectic structure} on a $(2n+1)$-dimensional manifold $M$ is a pair $(\alpha, \omega)$ of closed forms, where $\alpha \in \Omega^1(M)$ and $\omega \in \Omega^{2n}(M)$, such that $\alpha \wedge \omega$ is a volume form.
This generalizes \emph{cosymplectic structures}, where \( \omega = \beta^{n} \) for a closed \( \beta \in \Omega^{2}(M) \).
This structure contrasts with contact forms, where the somehow opposite maximal non-integrability is required. While contact structures exist on all closed 3-manifolds, cosymplectic structures impose topological constraints, such as nontrivial first Betti number, which have been studied in detail in \cite{CLM}, and are fully characterized in the compact case in \cite{GMP1}. Cosymplectic structures also arise naturally as critical hypersurfaces in $b$-symplectic or log symplectic manifolds \cite{GMP2}, where they model degenerate limits of symplectic structures, often in connection with manifolds with boundary or partial compactifications.

The following elementary lemma, which is crucial for the proof of Theorem~\ref{T:main}, shows that there is a correspondence between harmonic vector fields and weak cosymplectic structures in any dimension.

\begin{lemmaIntro}
    [The cosymplectic correspondence]\label{thm:mainCorrespondenceResult}
There exists a correspondence between pairs $(X, g)$ of nonvanishing harmonic vector fields and Riemannian metrics on a manifold $M$, and weak cosymplectic structures $(\alpha, \omega)$ satisfying $\alpha = \iota_X g$.
\end{lemmaIntro}

Furthermore, by means of this geometric link, the harmonic field $X$ corresponds, up to a proportionality factor, to the \emph{Reeb vector field} of the weak cosymplectic structure $(\alpha, \omega)$, i.e., the unique vector field $Y$ satisfying $\alpha(Y) = 1$ and $\omega(Y, -) =0$, where $Y=\frac{X}{g(X,X)}$.

The article is organized as follows. In Section~\ref{sec:correspondence}, we establish the correspondence between harmonic vector fields and weak cosymplectic structures, providing a detailed proof of Lemma~\ref{thm:mainCorrespondenceResult}. In Section~\ref{sec:Turing}, we review the fundamentals of Turing completeness and its relation with dynamics. Section~\ref{sec:NavierStokesApplication} presents the application to the Navier-Stokes equations, including the construction of a Turing complete flow and a discussion of its computational implications.

\section{Proof of the cosymplectic correspondence - Lemma \ref{thm:mainCorrespondenceResult}}
\label{sec:correspondence}

Recall that a Riemannian metric on a manifold $M$ of dimension $2n+1$ induces the so-called Hodge star operator $\hodgestar \colon \Omega^k(M) \to \Omega^{2n+1-k}(M)$ for all $0 \leq k \leq 2n+1$. In terms of this map, the codifferential $d^* \colon \Omega^{k}(M) \to \Omega^{k-1}(M)$ can be explicitly written as $d^* = (-1)^{k} \hodgestar d \hodgestar$. In particular, since $M$ is compact and without boundary, this implies that a differential form $\alpha$ is harmonic for the Hodge Laplacian $\Delta = dd^*+d^*d$ if and only if both $\alpha$ and $\hodgestar\alpha$ are closed.

Now, consider a pair $(X, g)$ of a nonvanishing harmonic vector field $X$ and a Riemannian metric $g$ on $M$. This means that the $1$-form $\alpha=\iota_Xg$ is harmonic for the Hodge Laplacian, where $\iota_X$ denotes the contraction of the first index by $X$. Then, the pair $(\alpha, \hodgestar \alpha)$ defines a weak cosymplectic structure. Indeed, since $\alpha$ is harmonic, both $\alpha \in \Omega^{1}(M)$ and $\hodgestar \alpha \in \Omega^{2n}(M)$ are closed, and also 
$$\alpha \wedge \hodgestar \alpha=g(X, X)\,\mu$$ 
is proportional to the Riemannian volume form $\mu$, via a positive proportionality factor.

For the converse implication, let $(\alpha, \omega)$ be a weak cosymplectic structure with Reeb vector field \( X \).
Since $\alpha \wedge \omega$ is nonzero, we have that $\ker(\omega) \cap \ker(\alpha) = 0$, so in particular $\omega$ defines a volume form on $\ker(\alpha)$.
Pick a metric on $\ker(\alpha)$ whose Riemannian volume is $\omega$, and then extend it to a metric $g$ on the whole tangent space by setting {$g(X, X)=1$ and imposing that $\ker(\alpha)$ and $X$ are orthogonal}.
Notice that, with these conditions, the Riemannian volume of $g$ is precisely $\alpha \wedge \omega$ {and that \( \alpha^{\sharp} = X \)}.



With this metric, observe that
$$
    \hodgestar \alpha = \hodgestar (1 \wedge \alpha) = \iota_{\alpha^\sharp} (\hodgestar 1) = {\iota_{{\alpha^\sharp}}} {(\alpha \wedge \omega)} = {\alpha({\alpha^{\sharp}})} \wedge \omega +  {\alpha}\wedge {\iota_{\alpha^\sharp}}\omega = \omega.
$$
This implies that $d\hodgestar \alpha = d \omega = 0$. Hence, since $d\alpha = 0$ by hypothesis, then $X = \alpha^{\sharp}$ is harmonic, concluding the proof.

\begin{remark}
	By Tischler’s theorem, the existence of a nonvanishing harmonic \(1\)-form \(\alpha\) forces \(M\) to fiber smoothly over \(\mathbb S^1\).
\end{remark}

\section{Turing Complete Dynamical Systems}\label{sec:Turing}

\subsection{Review of the Theory of Turing Machines}

A \emph{Turing machine} is a theoretical model of computation that formalizes the idea of an automatic calculation. Informally, a Turing machine is a read/write mechanism that operates on a bi-infinite tape serving as memory, where the tape is divided into discrete cells. Each cell contains a symbol from an alphabet, which can be taken to be binary. The machine is equipped with a head positioned over a single cell at any given time. In each computational step, the head reads the symbol in the current cell, and based on this symbol and its internal state, it performs three actions: it rewrites the symbol in the cell, moves one cell to the left or right, and updates its internal state. This process repeats either until a designated halting state is reached, at which point the computation terminates, or it continues indefinitely.

In formal terms, let us fix an alphabet \(\Sigma = \{0,1\}\) of binary symbols. A (binary) Turing machine is a quadruple $T=(Q, q_0, q_{\text{halt}}, \delta)$ where:
\begin{itemize}
    \item \( Q \) is a finite set of states,
    \item \( q_0 \in Q \) is the designated initial state,
    \item \( q_{\text{halt}} \in Q \), with \( q_{\text{halt}} \neq q_0 \), is the halting state,
    \item \( \delta \colon Q \times \Sigma \rightarrow Q \times \Sigma \times \{-1,0,1\} \) is the transition function.
\end{itemize}

The \emph{space of configurations} is \(Q \times \Sigma^* \), where $\Sigma^*$ is the set of two-sided sequences \( t = (t_i)_{i \in \mathbb{Z}} \) with values in \( \Sigma = \{0,1\}\) such that $t_i = 1$ for finitely many indices $i$. In this way, a configuration is a pair $c =(q, t)$ for which $q \in Q$ is the current state of the machine and $t \in \Sigma^*$ represents the state of the \emph{tape} of the machine. The initial configuration is of the form \( (q_0, t^{\text{in}}) \), where the tape $t^{\text{in}}$ contains the input data to the Turing machine.

The evolution of a Turing machine follows these rules:
\begin{enumerate}
    \item If the current configuration is of the form \((q_{\text{halt}}, t) \), the computation halts and $t$ is returned as output.
    \item Otherwise, if the configuration is $(q, t)$, the machine reads the symbol \( t_0 \) at the initial tape position and computes  
    \[
    \delta(q, t_0) = (q', t_0', \varepsilon).
    \]
    With this information, the machine writes the new symbol \( t_0' \) on the tape, shifts it by \( \varepsilon \in \{-1, 0, 1\} \) (where by convention, +1 denotes a left shift), and transitions to state \( q' \). The new configuration becomes \( (q', t') \), and the process repeats.
\end{enumerate}

This iterative process of a Turing machine $T$ defines a discrete dynamical system on the space of configurations, known as the \emph{global transition function}
\[
\Delta_T \colon Q \times \Sigma^* \rightarrow Q \times \Sigma^*,
\]
mapping a configuration to its immediate successor under the machine dynamics.

One of the foundational results in the theory of computation, proved by Alan Turing in 1936, is the \emph{undecidability of the halting problem}: there exists no algorithm that can determine whether an arbitrary Turing machine eventually reaches a configuration with halting state starting from a given input.

An important concept for our purposes is the \emph{universal Turing machine}. Notice that the information of the quadruple $T = (Q, q_0, q_{\text{halt}}, \delta)$ corresponding to a Turing machine can be encoded in binary as a finite sequence $t_T \in \Sigma^*$. Given a pair of sequences $t, t' \in\Sigma^*$, let us denote by $t*t'$ their juxtaposition.

\begin{definition}
A Turing machine \( T_U \) is said to be \emph{universal} if, for any other Turing machine \( T \) and initial tape \(t^{\text{in}} \), if we run $T_U$ on the tape $t_T * t^{\text{in}}$, then $T_U$ halts if and only if $T$ halts on $t^{\text{in}}$. Furthermore, if $T_U$ halts, the resulting tape agrees with the final tape computed by $T$ on $t^{\text{in}}$.
\end{definition}

Universal Turing machines exist and can be explicitly described (see \cite{moorebook}). As a direct consequence of the undecidability of the halting problem, any universal Turing machine admits configurations for which the halting status is true or false but not provable.

\subsection{Turing Complete Dynamical Systems}

We now introduce a formal notion of \emph{Turing completeness} for dynamical systems, capturing the idea that the system can simulate the computational behaviour of a universal Turing machine. Let \( X \) denote a dynamical system (either discrete or continuous) on a topological space \( M \).

\begin{definition}\label{TC}
A dynamical system \( X \) on \( M \) is said to be \emph{Turing complete} if there exists a universal Turing machine \( T_U \) such that for each initial configuration \( c \) of \( T_U \) and each tape $t \in \Sigma^*$, there exists a computable point \( p_{c} \in M \) and a computable open set \( U_{c,t} \subset M \) satisfying:
\[
T_U \text{ halts on } c \text{ with output $t$} \quad \Longleftrightarrow \quad \text{the positive trajectory of } X \text{ through } p_c \text{ intersects } U_{c,t}.
\]
\end{definition}

In this setting, the question of whether a Turing machine halts becomes equivalent to whether a particular trajectory in \( M \) intersects a prescribed region. Crucially, both \( p_c \) and \( U_{c,t} \) must be computable in a concrete sense: for example, if \( M \) is a smooth manifold, computability of \( p_c \) means its coordinates (in some chart) are explicitly determined from \( c \) and representable as rational numbers. Similarly, computability of \( U_{c,t} \) means it can be approximated effectively to any desired precision.

It follows that Turing complete dynamical systems necessarily exhibit undecidable behaviour: for some initial conditions, long-term dynamics are unpredictable in a rigorous computational sense.

In foundational work, Moore \cite{Mo1,moore2} constructed explicit low‐dimensional analytic vector fields on \(\R^3\), given by closed‐form differential equations inspired by idealized mechanical models, in which the flow carries out the state transitions of a universal Turing machine, thereby yielding trajectories whose reachability and long‐term behaviour are rigorously undecidable.  Building on this insight, Koiran and Moore \cite{koiran} subsequently exhibited concrete closed‐form analytic maps in one and two dimensions—using elementary analytic functions to encode tape symbols and head‐movement rules—whose iterates simulate any Turing computation, proving that even planar analytic dynamics suffice for full computational universality.  

It is worth mentioning that there exists a useful shortcut to prove Turing completeness of a smooth vector field $X$ on a $3$-dimensional manifold $M$ by means of Poincar\'e return maps. Recall that a \emph{Poincar\'e section} of $X$ is an embedded $2$-dimensional disk $D \subset M$ that is transverse to the flow of $X$ and nonwandering, i.e., for every $x \in D$ the flow of $x$ by $X$ hits $D$ again in positive time. A Poincar\'e section $D \subset M$ induces the so-called \emph{Poincar\'e return map} $D \to D$ given by the first hitting point of the flow after leaving $D$.

\begin{theorem}[{\cite[Theorem 5.2]{CMPP2}}]\label{thm:auxiliary}
    There exists an area preserving diffeomorphism $f \colon D \to D$, which is the identity on an open set of the boundary of $D$, such that, if $X$ is a vector field admiting a Poincar\'e section whose Poincar\'e return map is $f$, then $X$ is Turing complete.
\end{theorem}

\section{Proof of the main theorem - Theorem~\ref{T:main}}
\label{sec:NavierStokesApplication}

\subsection{Harmonic vector fields as stationary solutions to the Navier-Stokes equations}

The following result is standard, but we include a proof for the sake of completeness.

 \begin{proposition}\label{prop:harmonic-NS}
Let \((M,g)\) be a Riemannian \(3\)-manifold and let $X\in\mathfrak{X}(M)$
be a harmonic vector field.
Then $X$ is a stationary solution of the incompressible Navier-Stokes equations
for an appropriate choice of pressure \(p\).
\end{proposition}

\begin{proof}
Let us denote $\alpha = X^\flat = g(X, -)$. Since \(d \alpha= d^*\alpha =0\), we infer that
\[
\Delta X = [(d\,d^* + d^*\,d)\,\alpha]^\sharp = 0.
\]
Furthermore, notice that \(d^*\alpha=0\) also imples \(\operatorname{div} X=0\).

Plugging $X$ into the stationary Navier-Stokes equations, we obtain that it must satisfy
\[
\nabla_X X = -\,\nabla p.
\]
Finally, to show that \(\nabla_X X\) is a gradient, we note that \(d\alpha=0\) implies the following symmetry
\[
g\bigl(\nabla_Y X,\;Z\bigr)
=
g\bigl(\nabla_Z X,\;Y\bigr)
\quad\forall\,Y,Z.
\]
Hence for any \(Y\),
\[
Y\left(\tfrac12\,|X|^2\right)
= g\bigl(\nabla_Y X,\;X\bigr)
= g\bigl(\nabla_X X,\;Y\bigr),
\]
which shows
\(\nabla_X X = \nabla\bigl(\tfrac12\,|X|^2\bigr)\). Setting
\[
p = -\,\tfrac12\,|X|^2 + \text{constant}
\]
then yields \(\nabla_X X = -\,\nabla p\), which proves that \(X\) is a stationary solution of the Navier-Stokes equations.
\end{proof}

\subsection{Embedding dynamics}
Throughout this section, let is fix an area preserving diffeomorphism \( f \colon D \to D \) of the $2$-dimensional disk which is equal to the identity near the boundary.

\begin{lemma}\label{L:beta}
	There exists a \( 2 \)-form \( \beta \) on \( T = D \times \mathbb S^1 \) such that \( (c \dif t, \beta) \) is a cosymplectic structure whose time-\( c \) return map to \( D \times \{ 0 \} \) is \( f \).
	The form \( \beta \) is equal to \( \dif x \wedge \dif y \) near the boundary of the disk.
	\label{thm:torus-cosymplectic-structures-with-given-retur-map}
\end{lemma}

\begin{proof}
	Since \( f \) is area-preserving, we can form the cosymplectic manifold \(M_f = (D \times \mathbb{R})/(p, t+1)\sim(f(p), t)\) with cosymplectic structure \( (c \dif t, \dif x \wedge \dif y) \).
	The time-\( c \) return map of the Reeb vector field to the disk \( D \times \{ 0 \} \) is clearly~\( f \).

	Since \( f \) is isotopic to the identity, the manifold $M_f$ is diffeomorphic to \( D \times \mathbb S^1 \) by a diffeomorphism which is the identity near the boundary of the torus.
	Explicitly, if \( f_t \) is the isotopy with \( f_0 = f \) and \( f_1 = \identity \), the diffeomorphism is given by \( F(p, t) = (f_t^{-1}(p), t) \).

	We can now pull back the cosymplectic structure \( (c\dif t, \dif x \wedge \dif y) \) on $M_f$ to \( D \times \mathbb S^1 \) by \( F^{-1} \) to get a cosymplectic structure
	\( (c \dif t, \beta) \) on \( D \times \mathbb S^1\).
	Since \( f \) is the identity near the boundary, \( \beta = \dif x \wedge \dif y \) near the boundary.
\end{proof}

\begin{proposition}
	Let \( (M, \alpha, \beta) \) be a compact cosymplectic $3$-dimensional manifold.
	Then there exists a \( 2 \)-form \( \tilde{\beta} \), an embedded solid torus \( T = D \times \mathbb S^1 \subset M \), a smaller disk \( D_{0} \subset D \), and a constant \( c > 0 \) such that
	\begin{itemize}
		\item \( (\alpha, \tilde{\beta}) \) is a cosymplectic structure on \( M \).
		\item \( \tilde{\beta} \) is equal to \( \beta \) on the complement of \( T \).
		\item The time-\( c \) return map to the disk \( D_{0} \times \{ 0 \} \subset T \) of the Reeb vector field of the cosymplectic structure \( (\alpha, \tilde{\beta}) \) is equal to \( f \).
	\end{itemize}
	Furthermore, if \( (M, \alpha, \beta) \) is equipped with a metric \( g \) such that \( \beta = \hodgestar_{g}\alpha \), there exists a metric \( \tilde{g} \) such that \( \tilde\beta = \hodgestar_{\tilde g}\alpha \) and satisfying that \( g = \tilde{g} \) on the complement of \( T \).
	\label{thm:embed-diffeomorphism-free-cosymplectic-structure}
\end{proposition}

\begin{remark}
	The assumption that \( M \) is compact in \cref{thm:embed-diffeomorphism-free-cosymplectic-structure} can be weakened.
	Indeed, the compactness is only used in the proof of \cref{thm:cosymplectic-flow-torus} to obtain a circle transverse to \( \ker\alpha \).
	It is thus sufficient to assume that such a transverse circle exists. When \( M \) is compact any point has such a circle passing through it, so in this case one can choose \( T \) to pass through any given point.
	\label{rem:taut-sufficient-M-compact-not-necessary}
\end{remark}

We will need the following \namecref{thm:cosymplectic-flow-torus} for the proof of \Cref{thm:embed-diffeomorphism-free-cosymplectic-structure}.

\begin{lemma}
	Let \( (M, \alpha, \beta) \) be a compact cosymplectic $3$-dimensional manifold.
	Then there exists an embedded solid torus \( \iota \colon T = D \times \mathbb S^1 \to M \) such that \( \iota^* \alpha = c \dif t\) for some constant \( c \neq 0 \).
	\label{thm:cosymplectic-flow-torus}
\end{lemma}

\begin{proof}
	By \Cref{thm:mainCorrespondenceResult}, the $1$-form \( \alpha \) is intrinsically harmonic, i.e., it is harmonic for some Riemannian metric.
	Since \( M \) is compact, this implies, by \cite[Thm.~9.11]{farberTopologyClosedOneForms2004}, that there exists an embedded \( \mathbb S^1 \) which is everywhere transverse to \( \ker\alpha \).
	Actually, that Theorem implies that such a circle exists through any given point.
	Since transversality is an open condition and \( \mathbb S^1 \) is compact we can choose a small tubular neighbourhood \( T = D \times \mathbb S^1 \) on which the circle \( \{ q \} \times \mathbb S^1 \) is positively transverse to \( \ker\alpha \) for each $q\in D$.

	Now, the cohomology group \( H^1(T) \) is one-dimensional and generated by the form \( \dif t \) where \( t \) is the \( \mathbb S^1 \) coordinate, so we can write (the restriction of) \( \alpha \) as \( c \dif t + \dif g \) for some function \( g \colon T \to \mathbb{R} \).

	We claim that the map \( G(q, t) = (q, t + c^{-1}g(q, t))\) is a diffeomorphism of \( T \). Notice first that the map is well-defined; since \( g \) is a function on \( T \) it satisfies \( g(p, t+1) = g(p, t) \) and thus
	\begin{equation*}
		G(p, t+1)
		=
		(p, t + 1 + c^{-1} g(p, t+1))
		=
		(p, t + 1 + c^{-1} g(p, t))
		\sim
		(p, t + g(p, t))
		=
		G(p, t)
		.
	\end{equation*}
    Next, $G$ is a local diffeomorphism because 
    \[
    \text{Det}(\textup{D}G)=1+c^{-1}\partial_t g(q,t)>0,
    \]
    where we have used that $\alpha(\partial_t)=c+\partial_tg>0$ by transversality.
	Moreover, for each $q\in D$, the map $t+c^{-1}g(q,t)$ is locally injective because its derivative is positive and it is onto $\mathbb S^1$ because its image is the whole unit interval $[c^{-1}g(q,0),1+c^{-1}g(q,0)]\sim [0,1]$; accordingly, the map is a bijection. 
	The claim then follows from the fact that any bijective local diffeomorphism between compact manifolds is a global diffeomorphism.
    
    Notice that \( G^{*}(c \dif g) = c \dif (G^{*} t) = c \dif t + \dif g \).
	This means that a coordinate change of \( T \) by \( G \) gives us the embedded torus whose existence is claimed in the lemma, and we are done.
\end{proof}

\begin{proof}[Proof of \Cref{thm:embed-diffeomorphism-free-cosymplectic-structure}]
	Take the embedded torus \( T \subset M \) from Proposition~\ref{thm:cosymplectic-flow-torus} on which \( \alpha = c \dif t \).
	Choose a smaller subtorus \( T_{0} \subsetneq T \) and equip it with the \( 2 \)-form \( \beta' \) obtained from Lemma~\ref{L:beta}.
	Take polar coordinates \( (r, \theta, t) \) on \( T \) and choose a non-negative function \( \rho \colon M \to \R \) such that \( \rho\vert_{T_{0}} = 0 \), \( \rho\vert_{M \setminus T} = 1 \) and \( \rho \) on \( T \) only depends on the \( r \)-coordinate with \( \rho'(r) \geq 0 \).

	We can extend \( \beta' \) to a  \( 2 \)-form on \( M \), which we also denote by \( \beta' \), by setting it equal to \( 0 \) on \( M \setminus T \) and \( (1-\rho)\dif x \wedge \dif y \) on \( T \setminus T_{0} \).
	Since \( \beta' \) is equal to \( \dif x \wedge \dif y \) near the boundary of \( T_{0} \) this is clearly a smooth form.
	Note that this form is closed since \( \dif \beta' = -\frac{d}{dr}\rho(r) \dif r \wedge \dif x \wedge \dif y = 0 \).
	It also satisfies \( \alpha \wedge \beta' \geq 0 \) with equality only at points where \( \rho = 1 \).

	Now, it is convenient to introduce a second toroidal domain $T_1$, which is a neighborhood of the closure of $T$.
	Since \( H^{2}(T_{1}) = 0 \), the pullback of \( \beta \) to \( T_{1} \) is equal to \( \dif \eta \) for some \( 1 \)-form \( \eta \) on \( T_{1} \).
	On \( T_{1} \) with a thin central torus around \( r = 0 \) removed, we can write \( \eta = \eta_{r}\dif r + \eta_{\theta} \dif \theta + \eta_{t} \dif t \).
	Set \( k := \min_{T_{1} \backslash T_0} \eta_{\theta} \) and define the form
	\begin{equation*}
		\tilde\beta
		:=
		\dif\big(
			\rho(r)(\eta - k\dif \theta)
		\big)
		+
		\beta'
		.
	\end{equation*}
	This is a smooth \( 2 \)-form on \( M \).
	Indeed, in \( T_1 \setminus T \), the second term vanishes and \( \rho(r) = 1 \) so the first term simply becomes \( \dif(\eta - k\dif \theta) = \beta \), which admits the global extension of $\beta$ itself.
	On the other hand, near \( r = 0 \), where \( \dif \theta \) is not defined, the first term vanishes since there \( \rho = 0 \).
	Moreover, it is a closed $2$-form because the first term is exact on \( T_1 \)  and equal to \( \beta \) outside of \( T_1 \), and the second term \( \beta' \) is closed.
	Lastly, let us calculate
	\begin{equation*}
	\begin{split}
		\alpha \wedge \tilde{\beta}
		&=
		\alpha \wedge \frac{d}{dr}\rho(r) \dif r \wedge (\eta + k\dif \theta)
		+
		\rho(r) \alpha \wedge \dif \eta
		+
		\alpha \wedge \beta'
		\\
		&=
		c
		\frac{d}{dr}\rho(r) (\eta_{\theta} - k)
		\dif r \wedge \dif \theta \wedge \dif t
		+
		\rho(r) \alpha \wedge \beta
		+
		\alpha \wedge \beta'
		>
		0
		,
	\end{split}
	\end{equation*}
	since the first term is non-negative because \( \eta_{\theta} \geq k \) and \( \frac{d}{dr}\rho(r) \geq 0 \), the second term is positive except when \( \rho = 0 \), and the last term is positive except when \( \rho = 1 \).

	Now let a compatible metric \( g \) be given.
	By \Cref{thm:mainCorrespondenceResult}, there exists a metric \( \tilde{g} \) making \( \alpha = \hodgestar_{\tilde g} \tilde{\beta} \).
	The proof of this result is obtained by gluing together metrics on \( \ker\alpha \) which satisfy that their volume equals \( \tilde\beta \).
	This means that \( \tilde g \) can be constructed such that \( g \) and \( \tilde g \) agree outside of \( T \) where \( \beta = \tilde \beta \). This completes the proof of the proposition.
\end{proof}

\subsection{Proof of Theorem~\ref{T:main}}

Let $(M,g)$ be a Hodge-admissible Riemannian 3-manifold.
By definition, there exists a nowhere vanishing harmonic 1-form $\alpha$ on $M$ or, equivalently by

Lemma~\ref{thm:mainCorrespondenceResult}, the pair $(\alpha,\hodgestar\alpha)$ is a cosymplectic structure.

In this situation, Proposition~\ref{thm:embed-diffeomorphism-free-cosymplectic-structure} allows us to embed any area preserving diffeomorphism of the disk $f \colon D \to D$ which is the identity near the boundary as the time-$c$ return map on a Poincar\'e section $D \subset M$ of the Reeb field \( \tilde{X} \) of a cosymplectic structure that is a deformation of $(\alpha,\star\alpha)$ on a solid torus containing this disk. The proposition also gives a metric $\tilde g$, which coincides with $g$ in the complement of the aforementioned solid torus, so that $\alpha$ is harmonic with respect to the metric $\tilde g$, and hence the dual field $\tilde X$ is harmonic.

Choosing $f$ to simulate a universal Turing machine (cf.\ Theorem \ref{thm:auxiliary}) yields a harmonic field $\tilde X$ whose Poincaré map is $f$. By Proposition \ref{prop:harmonic-NS}, $\tilde X$ is thus a Turing complete stationary solution to the Navier-Stokes equations for any value of the viscosity $\nu\geq 0$. This completes the proof of the theorem.

\qed

\section*{Conclusion}

In this work, we have constructed stationary solutions to the Navier-Stokes equations on a wide class of Riemannian \(3\)-manifolds that simulate a universal Turing machine (i.e., they are Turing complete). These solutions are harmonic fields, and they are constructed using the machinery of cosymplectic geometry; being harmonic, we emphasize that they are also Euler steady states. 

Crucially, we construct vector fields that are solutions to the stationary Navier-Stokes equations for any value of the viscosity $\nu\geq 0$. The computational power is then not jeopardized because the fields are not perturbed when the internal friction is added, and therefore the fragility of Turing completeness under perturbations~\cite{Bournez} is not affected by the value of $\nu$.

Our construction yields stationary (i.e., time-independent) solutions. Whether such computational universality extends to genuinely time-dependent solutions remains an open and intriguing question. In the inviscid case, this was successfully accomplished in \cite{timedependent}, where time-dependent Turing-complete Euler flows were constructed.


\printbibliography

\end{document}